\documentclass[a4paper,11pt]{article}

\usepackage[utf8]{inputenc}
\usepackage[english,activeacute]{babel}
\usepackage{amsmath,amsthm,amsfonts,amssymb}
\usepackage{mathpartir}
\usepackage{ stmaryrd }
\usepackage{graphics}
\usepackage{enumerate}
\usepackage{mathrsfs}
\input{xy}
\usepackage[all]{xy}

\theoremstyle{plain}
\newtheorem{thm}{Theorem}[subsection]
\newtheorem{cor}[thm]{Corollary}
\newtheorem{lemma}[thm]{Lemma}

\newtheorem{rmk}[thm]{Remark}
\newtheorem{defs}[thm]{Definition}
\theoremstyle{remark}

\input{diagxy}
\xyoption{curve}

\newbox\anglebox % large pullback angle
\setbox\anglebox=\hbox{\xy \POS(75,0)\ar@{-} (0,0) \ar@{-} (75,75)\endxy}
 
\newbox\angleboxr % reverse large pullback angle
\setbox\angleboxr=\hbox{\xy \POS(0,0)\ar@{-} (0,75) \ar@{-} (75,0)\endxy}
 
\newbox\sanglebox % small pullback angle
\setbox\sanglebox=\hbox{\xy \POS(50,0)\ar@{-} (0,0) \ar@{-} (50,50)\endxy}
 
\newbox\sangleboxr % small reverse pullback angle
\setbox\sangleboxr=\hbox{\xy \POS(0,0)\ar@{-} (0,50) \ar@{-} (50,0)\endxy}
 
\newbox\sangleboxf % small flipped pullback angle
\setbox\sangleboxf=\hbox{\xy \POS(50,50)\ar@{-} (50,0) \ar@{-} (0,50)\endxy}
 
\newbox\angleboxf % flipped pullback angle
\setbox\angleboxf=\hbox{\xy \POS(75,75)\ar@{-} (75,0) \ar@{-} (0,75)\endxy}
 
\newbox\sangleboxfr % small flipped reverse pullback angle
\setbox\sangleboxfr=\hbox{\xy \POS(0,50)\ar@{-} (50,50) \ar@{-} (0,0)\endxy}
 
\newbox\angleboxfr % small flipped reverse pullback angle
\setbox\angleboxfr=\hbox{\xy \POS(0,75)\ar@{-} (75,75) \ar@{-} (0,0)\endxy}

\newcommand{\alg}[1]{\ensuremath{\mathbf{#1}}}

\newcommand{\Sets}{\ensuremath{\mathbf{Set}}}

\newcommand{\theory}{\ensuremath{\mathbb{T}}}
\title{COMPLETENESS OF INFINITARY HETEROGENEOUS LOGIC}
\author{Christian Esp\'indola}

\begin{document}
\date{}
\maketitle

\begin{abstract}
Given a regular cardinal $\kappa$ such that $\kappa^{<\kappa}=\kappa$ (e.g. if the Generalized Continuum Hypothesis holds), we develop a proof system for classical infinitary logic that includes heterogeneous quantification (i.e., infinite alternate sequences of quantifiers) within the language $\mathcal{L}_{\kappa^+, \kappa}$, where there are conjunctions and disjunctions of at most $\kappa$ any formulas and quantification (including the heterogeneous one) is applied to less than $\kappa$ many variables. This type of quantification is interpreted in $\mathcal{S}et$ using the usual second-order formulation in terms of strategies for games, and the axioms are based on a stronger variant of the axiom of determinacy for game semantics. Within this axiom system we prove the soundness and completeness theorem with respect to a class of set-valued structures that we call well-determined. Although this class is more restrictive than the class of determined structures in Takeuti's determinate logic, the completeness theorem works in our case for a wider variety of formulas of $\mathcal{L}_{\kappa^+, \kappa}$, and the category of well-determined models of heterogeneous theories is accessible. Our system is formulated within the sequent style of categorical logic and we do not need to impose any specific requirements on the proof trees, disregarding thus the eigenvariable conditions needed in Takeuti's system. We also investigate intuitionistic systems with heterogeneous quantifiers for $\mathcal{L}_{\kappa^+, \kappa, \kappa}$ (when only conjunctions of less than $\kappa$ many formulas are allowed), and prove analogously a completeness theorem with respect to well-determined structures in categories in general, in $\kappa$-Grothendieck toposes in particular, and, when $\kappa^{<\kappa}=\kappa$, also in Kripke models. Finally, we consider an extension of our system in which heterogeneous quantification with bounded quantifiers is expressible, and extend our completeness results to that case.
\end{abstract}

\noindent $\mathbf{Keywords:}$ heterogeneous quantifiers, infinitary logics, game semantics, determinacy, completeness theorems.

\section{Introduction}

This paper is a continuation of the investigation begun in \cite{espindola} on infinitary categorical logic, focusing now on extending the completeness theorems for the homogeneous fragment of $\mathcal{L}_{\kappa^+, \kappa}$ to a system which includes heterogeneous quantification on less than $\kappa$ many variables.

Heterogeneous quantifiers (infinite alternations of universal and existential quantification) present a new kind of quantification in infinitary logic related to game semantics, in which two players successively chose elements of a structure and their goal is to satisfy (respectively falsify) a certain formula when evaluated in those elements. Classical proof systems for heterogeneous quantification have been introduced by Takeuti (see \cite{takeuti1} and chapter $4$ of \cite{takeuti2}) based on the axiom of determinacy, according to which in every such game one of the players has a winning strategy (for an account of this type of games see \cite{gs}). Takeuti's system is expressed via Gentzen-type sequents of the form $\Gamma \vdash \Lambda$, where $\Gamma$ and $\Lambda$ are sets of at most $\kappa$ many formulas in $\mathcal{L}_{\kappa, \kappa}$ (where $\kappa^{<\kappa}=\kappa$). This amounts to working with a special type of formulas of $\mathcal{L}_{\kappa^+, \kappa}$, namely the ones of the form:

$$\bigwedge_{i<\kappa} \phi_i \to \bigvee_{i<\kappa} \psi_i \qquad (1)$$
\noindent
where the $\phi_i$ and $\psi_i$ are arbitrary formulas in $\mathcal{L}_{\kappa, \kappa}$. 

As shown in \cite{takeuti1}, every such formula that is valid in all determinate structures (i.e., all structures that satisfy the axiom of determinacy), is provable. 
Our system differs from that of Takeuti essentially in the fact that, in principle, we will be able to derive provability from the validity of any formula of $\mathcal{L}_{\kappa^+, \kappa}$ rather than for the ones of the special form $(1)$. To do so, we will work within an axiom system that involves a stronger form of determinacy; more precisely, we will require to work in structures for which every game definable by formulas belonging to a certain class $\mathcal{C}$ satisfies the following two conditions:

\begin{itemize}
\item one of the players has a winning strategy;
\item if a given player, for every $n \in \alpha$ has still a winning strategy after having played the first $n$ moves of a given sequence $(x_n)_{n \in \alpha}$, then the sequence belongs to the set of winning games for that player.
\end{itemize}

Structures that satisfy both requirements above will be called \emph{well-determined}. If we take as $\mathcal{C}$ the class of all formulas in $\mathcal{L}_{\kappa^+, \kappa}$, the second condition above is too strong to be of use, since we will see that a structure that satisfies it for every possible definable game consists of just one element. This degenerate case is of little interest, since the meaning of universal and existential quantification agree, and hence no extra expressivity is gained with heterogeneous quantification. For this reason, it will be convenient to work with classes $\mathcal{C}$ for which the variety of structures satisfying both conditions above is richer. For such $\mathcal{C}$, we can consider the subclass of formulas in $\mathcal{L}_{\kappa^+, \kappa}$ where heterogeneous quantification is only applied to formulas of $\mathcal{C}$. This will include naturally some formulas of the form $(1)$, but also many others that do not take such a form. For all of them, we will see that validity in the class of structures well-determined for $\mathcal{C}$ will be equivalent to provability from our particular set of axioms. Moreover, the way we will achieve this completeness result will show that from the axiom schemata involving heterogeneous quantification that are used to prove valid formulas, only instantiations in formulas of $\mathcal{C}$ appear in such a proof. Therefore, for this class $\mathcal{C}$ we can set up an axiomatic system that is sound and complete with respect to well-determined structures for $\mathcal{C}$, as long as heterogeneous quantification is only applied to formulas in $\mathcal{C}$. Of course, to make this interesting, we will check that we have a good supply of structures well-determined for $\mathcal{C}$.

\subsection{$\kappa$-heterogeneous logic}

Let $\kappa$ be a regular cardinal such that $\kappa^{<\kappa}=\kappa$. The syntax of $\kappa$-heterogeneous logic consists of a (well-ordered) set of sorts and a set of function and relation symbols, these latter together with the corresponding type, which is a subset with less than $\kappa$ many sorts. Therefore, we assume that our signature may contain relation and function symbols on $\gamma<\kappa$ many variables, and we suppose there is a supply of $\kappa$ many fresh variables of each sort. Terms and atomic formulas are defined as usual, and general formulas are defined inductively according to the following:

\begin{defs} If $\phi, \psi, \{\phi_{\alpha}: \alpha<\delta\}$ (for each $\delta<\kappa^+$) are $\kappa$-heterogeneous formulas, then for each $\gamma<\kappa$ the following are also formulas: $(\forall\exists)_{\alpha<\gamma} \mathbf{x_{\alpha}} \phi$, $(\exists\forall)_{\alpha<\gamma} \mathbf{x_{\alpha}} \phi$, $\exists_{\alpha<\gamma} \mathbf{x_{\alpha}} \phi$, $\forall_{\alpha<\gamma} \mathbf{x_{\alpha}} \phi$, $\phi \to \psi$, and $\bigwedge_{\alpha<\eta}\phi_{\alpha}$ (where $\eta<\kappa$) and $\bigvee_{\alpha<\delta}\phi_{\alpha}$, these latter provided that $\cup_{\alpha<\delta}FV(\phi_{\alpha})$, the set of free variables of all $\phi_{\alpha}$, has cardinality less than $\kappa$.
\end{defs}

In this definition, the bold $\mathbf{x_{\alpha}}$ represents a block of variables, while $x$ represents a single variable. The intended meaning of a heterogeneous quantifier $(\forall\exists)_{\alpha<\gamma} \mathbf{x_{\alpha}} \phi$ is its formal expansion as $\forall \mathbf{x_0} \exists \mathbf{x_1} \forall \mathbf{x_2} ... \phi(\mathbf{x_0}, \mathbf{x_1}, \mathbf{x_2}, ...)$. Unlike homogeneous formulas, this intended meaning is non well-founded, as in particular we can have as a subformula of this expansion also the expansion of each $(\forall\exists)_{\beta \leq \alpha<\gamma} \mathbf{x_{\alpha}} \phi$ for $\beta<\gamma$, with the obvious meaning of the notation that we now make explicit in the following convention. Call an ordinal $\beta$ even (respectively, odd) if it is of the form $\alpha+n$ with $\alpha$ limit and $n \in \omega$ and $n$ is even (respectively, odd); we also denote $\gamma - \beta$ the unique ordinal isomorphic to $\gamma \setminus \beta$. We adopt the following notation:

\begin{enumerate}
\item $(\forall\exists)_{\beta \leq \alpha<\gamma} \mathbf{x_{\alpha}} \phi :=\begin{cases}
(\forall\exists)_{\alpha<\gamma - \beta} \mathbf{x_{\beta+\alpha}} \phi & \text{ if $\beta$ is even}\\
(\exists\forall)_{\alpha<\gamma - \beta} \mathbf{x_{\beta+\alpha}} \phi & \text{ if $\beta$ is odd}
\end{cases}$
\item $(\forall\exists)_{\beta<\alpha<\gamma} \mathbf{x_{\alpha}} \phi := (\forall\exists)_{\beta+1 \leq \alpha<\gamma} \mathbf{x_{\alpha}} \phi \text{ and }(\forall\exists)_{\emptyset} \mathbf{x_{\alpha}} \phi :=\phi$.
\end{enumerate}
 \noindent
as well as the dual definitions for $(\exists\forall)_{\beta \leq \alpha<\gamma} \mathbf{x_{\alpha}} \phi$, etc.

We use sequent style calculus to formulate the axioms of $\kappa$-heterogeneous logic, as can be found, e.g., in \cite{johnstone}, D1.3. Each sequent $\phi \vdash_{\mathbf{x}} \psi$ has a context $\mathbf{x}$ consisting of less than $\kappa$ many variables. The system for $\kappa$-heterogeneous logic is described in the following definition. Besides the specific axioms for heterogeneous quantification, it features the transfinite transitivity rule, which was explained in detail in \cite{espindola1} (under the name ``rule T"). The axioma schemata for the heterogeneous quantification axiomatize the structures satisfying the two game-theoretic conditions imposed in the introduction.

\begin{defs}\label{sfol}
 The system of axioms and rules for $\kappa$-heterogeneous logic consists of

\begin{enumerate}
 \item Structural rules:
 \begin{enumerate}
 \item Identity axiom:
\begin{mathpar}
\phi \vdash_{\mathbf{x}} \phi 
\end{mathpar}
\item Substitution rule:
\begin{mathpar}
\inferrule{\phi \vdash_{\mathbf{x}} \psi}{\phi[\mathbf{s}/\mathbf{x}] \vdash_{\mathbf{y}} \psi[\mathbf{s}/\mathbf{x}]} 
\end{mathpar}

where $\mathbf{y}$ is a string of variables including all variables occurring in the string of terms $\mathbf{s}$.
\item Cut rule:
\begin{mathpar}
\inferrule{\phi \vdash_{\mathbf{x}} \psi \\ \psi \vdash_{\mathbf{x}} \theta}{\phi \vdash_{\mathbf{x}} \theta} 
\end{mathpar}
\end{enumerate}

\item Equality axioms:

\begin{enumerate}
\item 

\begin{mathpar}
\top \vdash_{x} x=x 
\end{mathpar}

\item 

\begin{mathpar}
(\mathbf{x}=\mathbf{y}) \wedge \phi[\mathbf{x}/\mathbf{w}] \vdash_{\mathbf{z}} \phi[\mathbf{y}/\mathbf{w}]
\end{mathpar}

where $\mathbf{x}$, $\mathbf{y}$ are contexts of the same length and type and $\mathbf{z}$ is any context containing $\mathbf{x}$, $\mathbf{y}$ and the free variables of $\phi$.
\end{enumerate}

\item Conjunction axioms and rules:

$$\bigwedge_{i<\gamma} \phi_i \vdash_{\mathbf{x}} \phi_j$$

\begin{mathpar}
\inferrule{\{\phi \vdash_{\mathbf{x}} \psi_i\}_{i<\gamma}}{\phi \vdash_{\mathbf{x}} \bigwedge_{i<\gamma} \psi_i}
\end{mathpar}

for each cardinal $\gamma<\kappa^+$.

\item Disjunction axioms and rules:

$$\phi_j \vdash_{\mathbf{x}} \bigvee_{i<\gamma} \phi_i$$

\begin{mathpar}
\inferrule{\{\phi_i \vdash_{\mathbf{x}} \theta\}_{i<\gamma}}{\bigvee_{i<\gamma} \phi_i \vdash_{\mathbf{x}} \theta}
\end{mathpar}

for each cardinal $\gamma<\kappa^+$.

\item Implication rule:
\begin{mathpar}
\mprset{fraction={===}}
\inferrule{\phi \wedge \psi \vdash_{\mathbf{x}} \eta}{\phi \vdash_{\mathbf{x}} \psi \to \eta}
\end{mathpar}

\item Existential rule:
\begin{mathpar}
\mprset{fraction={===}}
\inferrule{\phi \vdash_{\mathbf{x} \mathbf{y}} \psi}{\exists \mathbf{y}\phi \vdash_{\mathbf{x}} \psi}
\end{mathpar}

where no variable in $\mathbf{y}$ is free in $\psi$.

\item Universal rule:
\begin{mathpar}
\mprset{fraction={===}}
\inferrule{\phi \vdash_{\mathbf{x} \mathbf{y}} \psi}{\phi \vdash_{\mathbf{x}} \forall \mathbf{y} \psi}
\end{mathpar}

where no variable in $\mathbf{y}$ is free in $\phi$.

\item Transfinite transitivity rule:

\begin{mathpar}
\inferrule{\phi_{f} \vdash_{\mathbf{y}_{f}} \bigvee_{g \in \gamma^{\beta+1}, g|_{\beta}=f} \exists \mathbf{x}_{g} \phi_{g} \\ \beta<\kappa, f \in \gamma^{\beta} \\\\ \phi_{f} \dashv \vdash_{\mathbf{y}_{f}} \bigwedge_{\alpha<\beta}\phi_{f|_{\alpha}} \\ \beta < \kappa, \text{ limit }\beta, f \in \gamma^{\beta}}{\phi_{\emptyset} \vdash_{\mathbf{y}_{\emptyset}} \bigvee_{f \in B}  \exists_{\beta<\delta_f}\mathbf{x}_{f|_{\beta +1}} \bigwedge_{\beta<\delta_f}\phi_{f|_{\beta+1}}}
\end{mathpar}
\\
for each cardinal $\gamma<\kappa^+$, where $\mathbf{y}_{f}$ is the canonical context of $\phi_{f}$, provided that, for every $f \in \gamma^{\beta+1}$,  $FV(\phi_{f}) = FV(\phi_{f|_{\beta}}) \cup \mathbf{x}_{f}$ and $\mathbf{x}_{f|_{\beta +1}} \cap FV(\phi_{f|_{\beta}})= \emptyset$ for any $\beta<\gamma$, as well as $FV(\phi_{f}) = \bigcup_{\alpha<\beta} FV(\phi_{f|_{\alpha}})$ for limit $\beta$. Here $B \subseteq \gamma^{< \kappa}$ consists of the minimal elements of a given bar over the tree $\gamma^{\kappa}$, and the $\delta_f$ are the levels of the corresponding $f \in B$.
\\
The following axioms are intended to be instantiated only on formulas $\phi$ belonging to a certain subclass $\mathcal{C}$.

\item Heterogeneous axioms:

$$(\forall\exists)_{\alpha<\gamma} \mathbf{x_{\alpha}} \phi \vdash_{\mathbf{x}}\forall \mathbf{x_0} (\forall\exists)_{0< \alpha<\gamma} \mathbf{x_{\alpha}} \phi$$
$$(\exists\forall)_{\alpha<\gamma} \mathbf{x_{\alpha}} \neg \phi \vdash_{\mathbf{x}}\exists \mathbf{x_0} (\exists\forall)_{0<\alpha<\gamma} \mathbf{x_{\alpha}} \neg \phi$$

\noindent
for each limit ordinal $\gamma<\kappa$, and

$$(\forall\exists)_{\alpha<\gamma} \mathbf{x_{\alpha}} \phi \vdash_{\mathbf{x}}(\forall\exists)_{\alpha<\beta} \mathbf{x_{\alpha}} (\forall\exists)_{\beta \leq \alpha<\gamma} \mathbf{x_{\alpha}} \phi$$
$$(\exists\forall)_{\alpha<\gamma} \mathbf{x_{\alpha}} \neg \phi \vdash_{\mathbf{x}}(\exists\forall)_{\alpha<\beta} \mathbf{x_{\alpha}} (\exists\forall)_{\beta \leq \alpha<\gamma} \mathbf{x_{\alpha}} \neg \phi$$

\noindent
for each pair of limit ordinals $\beta<\gamma<\kappa$.

\item Preservation axioms:

$$\bigwedge_{\beta<\gamma} (\forall\exists)_{\beta \leq \alpha<\gamma} \mathbf{x_{\alpha}} \phi (\mathbf{x}, \mathbf{y_0}, \mathbf{y_1}, ..., \mathbf{x_{\beta}}, \mathbf{x_{\beta+1}}, ...) \vdash_{\mathbf{x} \cup \mathbf{y}_{\gamma}} \phi (\mathbf{x}, \mathbf{y_0}, \mathbf{y_1}, ...)$$
$$\bigwedge_{\beta<\gamma} (\exists\forall)_{\beta \leq \alpha<\gamma} \mathbf{x_{\alpha}} \neg \phi (\mathbf{x}, \mathbf{y_0}, \mathbf{y_1}, ..., \mathbf{x_{\beta}}, \mathbf{x_{\beta+1}}, ...) \vdash_{\mathbf{x} \cup \mathbf{y}_{\gamma}} \neg \phi (\mathbf{x}, \mathbf{y_0}, \mathbf{y_1}, ...)$$

\noindent
for each limit ordinal $\gamma<\kappa$, where $\mathbf{y}_{\gamma}=\cup_{\alpha<\gamma}\mathbf{y}_{\alpha}$.

\item Axiom of determinacy

$$\top \vdash_{\mathbf{x}} (\forall\exists)_{\alpha<\gamma} \mathbf{x_{\alpha}} \phi \vee (\exists\forall)_{\alpha<\gamma} \mathbf{x_{\alpha}} \neg \phi$$
\noindent
for each limit ordinal $\gamma<\kappa$.

\end{enumerate}
\end{defs}

Note that we are not assuming the axiom of excluded middle in general, which would be implied by the relevant instances of the axiom of determinacy for an appropriate choice of the subclass $\mathcal{C}$. Also, the axiom of determinacy and the preservation axioms are the formalizations in the language of the two conditions imposed on well-determined structures.

\section{$\kappa$-heterogeneous categories}

\subsection{Heterogeneous quantification in $\kappa$-Grothendieck toposes}

We define $\kappa$-Grothendieck toposes as sheaf toposes over a site whose underlying category is regular, has $\kappa$-limits and stable unions of less than $\kappa^+$ many subobjects, and where the coverage satisfies the transfinite transitivity property, i.e., transfinite composites of covering families are covering (we refer to \cite{espindola} and \cite{espindola1} for the terminology).

In the category of sets, heterogeneous quantification is defined through the usual second-order interpretation; more precisely, the category of sets satisfies the formula (in the language of set theory) $(\forall\exists)_{\alpha<\gamma} \mathbf{x_{\alpha}} \phi(\mathbf{x_0}, \mathbf{x_1}, ...)$ if and only if:

\begin{itemize}
\item there are functions $\mathbf{x_1}=f_1(\mathbf{x_0}), \mathbf{x_3}=f_3(\mathbf{x_0}, \mathbf{x_1}, \mathbf{x_2})$, ..., such that we have $\forall \mathbf{x_0}, \mathbf{x_2}, ... \phi(\mathbf{x_0}, f_1(\mathbf{x_0}), \mathbf{x_2}, f_3(\mathbf{x_0}, \mathbf{x_1}, \mathbf{x_2}), ...) \qquad (2)$
\end{itemize}

It is clear that the set-theoretic axioms then allow to transform $(2)$ in a finitary formula of the language of set theory, and we say that (2) holds if this finitary formula holds. Note that the restating $(2)$ is possible since the category of sets satisfies the axiom of choice, but in a general topos heterogeneous quantification cannot be defined in terms of its second-order interpretation (even if the internal language of the topos supports this latter).

Nevertheless, any $\kappa$-Grothendieck topos $\mathcal{E}$ admits a natural interpretation of heterogeneous quantification. Given a limit ordinal $\gamma<\kappa$ and a $\kappa$-chain of length $\gamma$ (i.e., a diagram $F: \gamma^{op} \to \mathcal{E})$, with $F(\alpha)=A \times \Pi_{i < \alpha}A_i$ , its limit projection $f: A \times A_{\gamma}:=A \times \Pi_{i<\gamma}A_i \to A$ and a subobject $S \rightarrowtail A \times A_{\gamma}$, we proceed to define the subobject $(\forall\exists)_f(S) \rightarrowtail A$ in a functorial way with respect to the relevant subobject lattices. In the same way as the definition of the interpretation of quantifiers in the topos makes use of the corresponding type of quantification at the meta-theoretical level, so the definition of the interpretation of heterogeneous quantification is going to use a heterogeneous quantification in the metatheory (which in our case, being $ZFC$, is strong enough to support such quantification). Actually, we will need an even stronger type of quantification in the metatheory than just heterogeneous; given sets $(A_i)_{i<\gamma}$, we would need to express a quantification of the form $\forall \mathbf{x_0} \in A_0 \exists \mathbf{x_1} \in A_1(\mathbf{x_0}) \forall \mathbf{x_2} \in A_2(\mathbf{x_0}, \mathbf{x_1}) ... \phi (\mathbf{x_0}, \mathbf{x_1}, \mathbf{x_2}, ...)$, in which all quantifiers appear bounded by sets that depend on previous variables. This can simply be done by stipulating that:

\begin{itemize}
\item there are choice functions $f_1$ on the family $\{A_1(\mathbf{x_0})\}_{\mathbf{x_0} \in A_0}$, $f_3$ on the family $\{A_3(\mathbf{x_0}, \mathbf{x_1}, \mathbf{x_2})\}_{\mathbf{x_0} \in A_0, \mathbf{x_1} \in A_1(\mathbf{x_0}), \mathbf{x_2} \in A_2(\mathbf{x_0}, \mathbf{x_1})}$, ...,  such that we have $\forall \mathbf{x_0} \in A_0, \mathbf{x_2} \in A_2(\mathbf{x_0}, f_1(\mathbf{x_0})), ...\phi(\mathbf{x_0}, f_1(\mathbf{x_0}), \mathbf{x_2}, f_3(\mathbf{x_0}, f_1(\mathbf{x_0}), \mathbf{x_2}), ...) \qquad (3)$
\end{itemize}

Here, the set-theoretic axioms and the process of transfinite recursion allow to transform $(3)$ into a finitary formula of the language of set theory.

We will nevertheless express this formula in words in the form: ``for all $\mathbf{x_0}$ in $A_0$ there is a $\mathbf{x_1}$ in $A_1(\mathbf{x_0})$ such that for all $\mathbf{x_2}$ in $A_2(\mathbf{x_0}, \mathbf{x_1})$ there is a $\mathbf{x_3}$ in $A_3(\mathbf{x_0}, \mathbf{x_1}, \mathbf{x_2})$ such that [...] such that $\phi$ holds at $\mathbf{x_0}, \mathbf{x_1}, \mathbf{x_2}, \mathbf{x_3}, ...$. With this in mind, we now give the following:

\begin{defs}\label{hett}
The subsheaf of $(\forall\exists)_f(S) \rightarrowtail A$ is defined by the following specification: $c \in (\forall\exists)_f(S) (C)$ if and only if for every arrow $D_0 \to C$ and element $e_0 \in A_0(D_0)$ there is a covering family $\{D_1^{i_1} \to D_0\}_{i_1 \in I_1}$ and elements $e_{i_1} \in A_1(D_1^{i_1})$ such that for every $i_1 \in I_1$, every arrow $D_2^{i_1} \to D_1^{i_1}$ and every element $e_2 \in A_2(D_2^{i_1})$ there is a covering family $\{D_3^{i_3} \to D_2^{i_1}\}_{i_3 \in I_3}$ and elements $e_{i_3} \in A_3(D_3^{i_3})$ such that $[...]$ such that if $f_b$ is the limit projection $D_{\gamma}^b \to C$ along a branch $b: ...D_3^{i_3} \to D_2^{i_1} \to D_1^{i_1} \to D_0 \to C$ of the corresponding tree over $C$, then $A(f_b)(c) \in S(D_{\gamma}^b)$.
\end{defs} 

It is clear that Definition \ref{hett} provides a presheaf, but to check that it is good, we first need to verify that it gives actually a subsheaf. Given a covering family $\{C_i \to C\}$ and a matching family of elements $c_i \in (\forall\exists)_f(S) (C_i) \subseteq A(C_i)$, since $A$ is a sheaf we get an element $c \in A(C)$; we claim that $c \in (\forall\exists)_f(S) (C)$. Indeed, given any arrow $l: D_0 \to C$, and element $e_0 \in A_0(D_0)$, the pullbacks $\{{D'}_0^{i_0}=l^*(C_i) \to D_0\}$ give a covering family and $e_0$ induces elements $e'_{i_0} \in A_0({D'}_0^{i_0})$. Since the $c_i$ are in $(\forall\exists)_f(S) (C_i)$, for each arrow $\{{D'}_0^{i_0}=l^*(C_i) \to D_0\}$ together with the element $e'_{i_0}$ there is a covering family $\{D_1^{i_0j} \to {D'}_0^{i_0} \}$ and elements $e'_{i_0j} \in A_1(D_1^{i_0j})$ such that $[...]$ such that the limit projection $f_b: D_{\gamma}^b \to C_i$ along the branch $b$ satisfies $A(f_b)(c_i) \in S(D_{\gamma}^b)$. Then the composite covering $\{D_1^{i_0j} \to {D'}_0^{i_0} \to D_0 \}$ together with the elements $e'_{i_0j}$ and the rest of the data obtained in the definition of the branch $b$, witness that $c \in (\forall\exists)_f(S) (C)$. Finally, it is clear that the definition is functorial between subobject lattices.

Definition \ref{hett} immediately provides the following forcing condition within Kripke-Joyal semantics: for $\mathbf{\beta} \in A(C)$ and variables $\mathbf{x_{\alpha}}$ of type $A_{\alpha}$, we have the following: $C \Vdash (\forall\exists)_{\alpha<\gamma} \mathbf{x_{\alpha}} \phi(\mathbf{\beta}, \mathbf{x_0}, \mathbf{x_1}, ...)$ if and only if for every arrow $D_0 \to C$ and element $e_0 \in A_0(D_0)$ there is a covering family $\{D_1^{i_1} \to D_0\}_{i_1 \in I_1}$ and elements $e_{i_1} \in A_1(D_1^{i_1})$ such that for every $i_1 \in I_1$, every arrow $D_2^{i_1} \to D_1^{i_1}$ and every element $e_2 \in A_2(D_2^{i_1})$ there is a covering family $\{D_3^{i_3} \to D_2^{i_1}\}_{i_3 \in I_3}$ and elements $e_{i_3} \in A_3(D_3^{i_3})$ such that $[...]$ such that if $\pi_{\alpha}: D_{\gamma}^b \to D_{\alpha}$ are the limit projections, $D_{\gamma}^b \Vdash \phi(\mathbf{\beta} f_b, e_0 \pi_0, e_{i_1} \pi_1, e_2 \pi_2, e_{i_3} \pi_3, ...)$. 

In an entirely analogous way, we can define the subsheaf $(\exists\forall)_f(S) \rightarrowtail A$ and obtain an analogous forcing statement.

\subsection{Heterogeneous quantification in $\kappa$-Heyting categories}

When the topology on the site of the topos is subcanonical, so that we have a full and faithful embedding of the underlying category of the site into the topos, the latter forcing condition can be expressed entirely in terms of the underlying category of the site. Indeed, via Yoneda lemma we can identify elements of $A_{\alpha}(D_{\alpha})$ with arrows $D_{\alpha} \to A_{\alpha}$ in the underlying category. On the other hand, any $\kappa$-Heyting category $\mathcal{C}$ (a category with a structure corresponding to axioms $1-9$) admits an embedding into its topos of sheaves $\mathcal{S}h(\mathcal{C}, \tau)$ with the subcanonical topology given by jointly epic arrows of cardinality less than $\kappa^+$. This motivates the following:

\begin{defs}\label{heti}
Given a $\kappa$-Heyting category, a sequence of objects $(A_{\alpha})_{\alpha<\kappa}$, a limit ordinal $\gamma<\kappa$ and a subobject $S \rightarrowtail A \times A_{\gamma}=A \times \Pi_{i<\gamma}A_i$, we define the sets of subobjects of $A$, $T_{(\forall\exists)}^{\gamma, S}$ and $T_{(\exists\forall)}^{\gamma, S}$ as follows:

\begin{enumerate}
\item for $l: C \rightarrowtail A$, $C \in T_{(\forall\exists)}^{\gamma, S} \subseteq \mathcal{S}ub(A)$ if and only if for every arrow $D_0 \to C$ and $e_0: D_0 \to A_0$ there is a covering family $\{D_1^{i_1} \to D_0\}_{i_1 \in I_1}$ and arrows $ e_{i_1}: D_1^{i_1} \to A_1$ such that for every $i_1 \in I_1$, every arrow $D_2^{i_1} \to D_1^{i_1}$ and every arrow $e_2: D_2^{i_1} \to A_2$ there is a covering family $\{D_3^{i_3} \to D_2^{i_1}\}_{i_3 \in I_3}$ and arrows $e_{i_3}: D_3^{i_3} \to A_3$ such that $[...]$ such that if $\pi_{\alpha}: D_{\gamma}^b \to D_{\alpha}$ are the limit projections, then the arrow $(lf_b, e_0\pi_0, e_{i_1}\pi_1, e_2\pi_2, e_{i_3}\pi_3, ...): D_{\gamma}^b \to A \times A_{\gamma}$ factors through $S$.

\item for $l: C \rightarrowtail A$, $C \in T_{(\exists\forall)}^{\gamma, S} \subseteq \mathcal{S}ub(A)$ if and only if there is a covering family $\{D_0^{i_0} \to C\}_{i_0 \in I_0}$ and arrows $e_{i_0}: D_0^{i_0} \to A_0$ such that for every $i_0 \in I_0$, every arrow $D_1^{i_0} \to D_0^{i_0}$ and every arrow $e_1: D_1^{i_0} \to A_1$ there is a covering family $\{D_2^{i_2} \to D_1^{i_0}\}_{i_2 \in I_2}$ and arrows $e_{i_2}: D_2^{i_2} \to A_2$ such that for every $i_2 \in I_2$, every arrow $D_3^{i_2} \to D_2^{i_2}$ and $e_3: D_3^{i_2} \to A_3$ $[...]$ such that if $\pi_{\alpha}: D_{\gamma}^b \to D_{\alpha}$ are the limit projections, then the arrow $(lf_b, e_{i_0}\pi_0, e_1\pi_1, e_{i_2}\pi_2, e_3\pi_3, ...): D_{\gamma}^b \to A \times  A_{\gamma}$ factors through $S$.
\end{enumerate}

\end{defs}

\begin{defs}\label{het0}
A $\kappa$-Heyting category $\mathcal{H}$ is called $\kappa$-heterogeneous if there are objects $(A_{\alpha})_{\alpha<\kappa}$ such that whenever we have a limit ordinal $\gamma<\kappa$, a diagram $F: \gamma^{op} \to \mathcal{C}$ with $F(\alpha)=A \times \Pi_{i < \alpha}A_i$, the limit projection $f: A \times A_{\gamma} \to A$ and a subobject $S \rightarrowtail A \times A_{\gamma}$ corresponding to the interpretation in $\mathcal{H}$ of a formula in the class $\mathcal{C}$, the joins:

$$\bigvee_{C \in T_{(\forall\exists)}^{\gamma, S}}C$$

$$\bigvee_{C \in T_{(\exists\forall)}^{\gamma, \neg S}}C$$

exist in $\mathcal{S}ub(A)$.
\end{defs}

A $\kappa$-heterogeneous category supports, hence, the interpretation of heterogeneous quantification. More precisely, if the interpretation in a $\kappa$-heterogeneous category of the formulas in context $(\mathbf{x_{\alpha}}, \top)$ is given by $A_{\alpha}$, while the interpretation of a formula in context $(\mathbf{\mathbf{x} x_{\gamma}}, \phi)$ (where $\mathbf{x_{\gamma}}= \cup_{\alpha<\gamma}\mathbf{x_{\alpha}})$ is given by $[[\phi]]$, and $\phi \in \mathcal{C}$,  then we proceed to interpret the formulas $(\forall\exists)_{\alpha<\gamma}\mathbf{x_{\alpha}} \phi$ and $(\exists\forall)_{\alpha<\gamma}\mathbf{x_{\alpha}} \phi$ respectively as the joins $\bigvee_{C \in T_{(\forall\exists)}^{\gamma, [[\phi]]}}C$ and $\bigvee_{C \in T_{(\exists\forall)}^{\gamma, [[\phi]]}}C$ of Definition \ref{het0}.

\begin{rmk}
It is easy to check that in the case of $\mathcal{S}et$, the interpretation of heterogeneous quantification just given coincides with that of (2) above.
\end{rmk}

The interpretation just given for heterogeneous quantification commutes with pullback functors, as can be seen through the following:

\begin{lemma}
\textbf{(Beck-Chevalley condition for heterogeneous quantification)} Suppose that the objects $(A_{\alpha})_{\alpha<\kappa}$ witness that a given category is $\kappa$-heterogeneous for a certain class $\mathcal{C}$. If all squares in the diagram:

\begin{displaymath}
\xymatrix{
B \times B_{\gamma} \ar@{->}[dd]_{h} \ar@{->}@/^2.0pc/[rrrrrr]^{g} & ... \qquad \ar@{->}[rr]_{\pi_{B, B_0, B_1}}&  & B \times B_0 \times B_1 \ar@{->}[dd]^{k_1} \ar@{->}[rr]_{\pi_{B, B_0}} & & B \times B_0 \ar@{->}[dd]^{k_0} \ar@{->}[r]_{\pi_B} & B \ar@{->}[dd]^{k}\\
 & & & & & & \\
A \times A_{\gamma} \ar@{->}@/_2.0pc/[rrrrrr]_{f} & ... \qquad \ar@{->}[rr]^{\pi_{A, A_0, A_1}} & & A \times A_0 \times A_1 \ar@{->}[rr]^{\pi_{A, A_0}} & & A \times A_0 \ar@{->}[r]^{\pi_A} & A \\
}
\end{displaymath}
\noindent
are pullbacks, then the diagram:

\begin{displaymath}
\xymatrix{
\mathcal{S}ub(B \times B_{\gamma})  \ar@{->}[rr]^{(\forall\exists)_g} &  & \mathcal{S}ub(B) \\
  & & \\
\mathcal{S}ub_{\mathcal{C}}(A \times A_{\gamma}) \ar@{->}[uu]^{h^{-1}} \ar@{->}[rr]_{(\forall\exists)_f} & & \mathcal{S}ub(A) \ar@{->}[uu]_{k^{-1}}\\
}
\end{displaymath}
\noindent
commutes, where $\mathcal{S}ub_{\mathcal{C}}(A \times A_{\gamma})$ consists of those $S \rightarrowtail A \times A_{\gamma}$ corresponding to the interpretation of formulas in $\mathcal{C}$.
\end{lemma}

\begin{proof}
Let $S \in \mathcal{S}ub_{\mathcal{C}}(A \times A_{\gamma})$, let $P$ be the subobject $(\forall\exists)_f(S) \rightarrowtail A$, let $Q$ be the subobject $k^{-1}(P) \rightarrowtail B$ and let $R$ be the subobject $h^{-1}(S) \rightarrowtail B \times B_{\gamma}$. We need to prove that $Q=(\forall\exists)_g(R)$.

Assume first that $C \leq Q$ in $\mathcal{S}ub(B)$, and let us show that $C \in T_{(\forall\exists)}^{\gamma, R}$. Since the composite $kl: C \to B \to A$ factors through $P$, by definition we have:

\begin{itemize}
\item for every arrow $f_0: D_0 \to C$ and $e_0: D_0 \to A \times A_0$ with $\pi_A e_0=klf_0$, there is a covering family $\{f_1^{i_1}: D_1^{i_1} \to D_0\}_{i_1 \in I_1}$ and arrows $e_{i_1}: D_1^{i_1} \to A \times A_0 \times A_1$ with $\pi_{A, A_0}e_{i_1}=e_0f_1^{i_1}$ such that for every $i_1 \in I_1$, every arrow $f_2^{i_1}: D_2^{i_1} \to D_1^{i_1}$ and every arrow $e_2: D_2^{i_1} \to A_2$ with $\pi_{A, A_0, A_1}e_{2}=e_0f_1^{i_1}f_2^{i_1}$, $[...]$ such that if $\pi_{\alpha}: D_{\gamma}^b \to D_{\alpha}$ are the limit projections, then the arrow $(klf_b, e_0\pi_0, e_{i_1}\pi_1, e_2\pi_2, ...): D_{\gamma}^b \to A \times A_{\gamma}$ factors through $S. \qquad (4)$
\end{itemize}
 
What we need to prove is actually that:

\begin{itemize}
\item for every arrow $f_0: D_0 \to C$ and $e'_0: D_0 \to B \times B_0$ with $\pi_B e'_0=lf_0$, there is a covering family $\{f_1^{i_1}: D_1^{i_1} \to D_0\}_{i_1 \in I_1}$ and arrows $e'_{i_1}: D_1^{i_1} \to B \times B_0 \times B_1$ with $\pi_{B, B_0}e'_{i_1}=e'_0f_1^{i_1}$ such that for every $i_1 \in I_1$, every arrow $f_2^{i_1}: D_2^{i_1} \to D_1^{i_1}$ and every arrow $e'_2: D_2^{i_1} \to A_2$ with $\pi_{A, A_0, A_1}e'_{2}=e'_0f_1^{i_1}f_2^{i_1}$, $[...]$ such that if $\pi_{\alpha}: D_{\gamma}^b \to D_{\alpha}$ are the limit projections, then the arrow $(lf_b, e'_0\pi_0, e'_{i_1}\pi_1, e'_2\pi_2, ...): D_{\gamma}^b \to B \times B_{\gamma}$ factors through $R. \qquad (5)$
\end{itemize}

But it is easy to see that $(5)$ follows directly from $(4)$. Indeed, given $f_0: D_0 \to C$ and $e'_0: D_0 \to B \times B_0$, we can define  $e_0: D_0 \to A \times A_0$ as the composite $e_0=k_0e'_0$, and then by $(4)$  there is a covering family $\{f_1^{i_1}: D_1^{i_1} \to D_0\}_{i_1 \in I_1}$ and arrows $e_{i_1}: D_1^{i_1} \to A \times A_0 \times A_1$ with $\pi_{A, A_0}e_{i_1}=e_0f_1^{i_1}=k_0e'_0f_1^{i_1}$. At this point we invoke the universal property of the pullback and get hence induced morphisms $e'_{i_1}: D_1^{i_1} \to B \times B_0 \times B_1$ with $\pi_{B, B_0}e'_{i_1}=e'_0f_1^{i_1}$. Continuing in this manner we obtain successively covering families from $(4)$, and morphisms $ e'_0, e'_{i_1}, e'_2, ...$ that induce the morphism $e'_{\gamma} := (lf_b, e'_0\pi_0, e'_{i_1}\pi_1, e'_2\pi_2, ...): D_{\gamma}^b \to B \times B_{\gamma}$. Now by $(4)$ the morphism $e_{\gamma} := (klf_b, e_0\pi_0, e_{i_1}\pi_1, e_2\pi_2, ...): D_{\gamma}^b \to A \times A_{\gamma}$ factors through $S$. But $e_{\gamma}=he'_{\gamma}$ (since both give the same morphisms when composed with the projections $\pi_{\alpha}$). Therefore, by the universal property of the pullback, $e'_{\gamma}$ must factor through $R$, which proves that $C \in T_{(\forall\exists)}^{\gamma, R}$ as we wanted.

Conversely, let $C \in T_{(\forall\exists)}^{\gamma, R}$ in $\mathcal{S}ub(B)$ and let us prove that $C \leq Q$. By the universal property of the pullback, this will follow as soon as we prove that the composite $kl: C \to B \to A$ factors through $P$. By hypothesis, the following holds:

\begin{itemize}
\item for every arrow $f_0: D_0 \to C$ and $e'_0: D_0 \to B \times B_0$ with $\pi_B e'_0=lf_0$, there is a covering family $\{f_1^{i_1}: D_1^{i_1} \to D_0\}_{i_1 \in I_1}$ and arrows $e'_{i_1}: D_1^{i_1} \to B \times B_0 \times B_1$ with $\pi_{B, B_0}e'_{i_1}=e'_0f_1^{i_1}$ such that for every $i_1 \in I_1$, every arrow $f_2^{i_1}: D_2^{i_1} \to D_1^{i_1}$ and every arrow $e'_2: D_2^{i_1} \to A_2$ with $\pi_{A, A_0, A_1}e'_{2}=e'_0f_1^{i_1}f_2^{i_1}$, $[...]$ such that if $\pi_{\alpha}: D_{\gamma}^b \to D_{\alpha}$ are the limit projections, then the arrow $(lf_b, e'_0\pi_0, e'_{i_1}\pi_1, e'_2\pi_2, ...): D_{\gamma}^b \to B \times B_{\gamma}$ factors through $R. \qquad (6)$
\end{itemize}

We need to prove, instead, that:

\begin{itemize}
\item for every arrow $f_0: D_0 \to C$ and $e_0: D_0 \to A \times A_0$ with $\pi_A e_0=klf_0$, there is a covering family $\{f_1^{i_1}: D_1^{i_1} \to D_0\}_{i_1 \in I_1}$ and arrows $e_{i_1}: D_1^{i_1} \to A \times A_0 \times A_1$ with $\pi_{A, A_0}e_{i_1}=e_0f_1^{i_1}$ such that for every $i_1 \in I_1$, every arrow $f_2^{i_1}: D_2^{i_1} \to D_1^{i_1}$ and every arrow $e_2: D_2^{i_1} \to A_2$ with $\pi_{A, A_0, A_1}e_{2}=e_0f_1^{i_1}f_2^{i_1}$, $[...]$ such that if $\pi_{\alpha}: D_{\gamma}^b \to D_{\alpha}$ are the limit projections, then the arrow $(klf_b, e_0\pi_0, e_{i_1}\pi_1, e_2\pi_2, ...): D_{\gamma}^b \to A \times A_{\gamma}$ factors through $S. \qquad (7)$
\end{itemize}

Once more, a similar argument than before shows that $(7)$ follows from $(6)$. Indeed, given $f_0: D_0 \to C$ and $e_0: D_0 \to A \times A_0$ with $\pi_A e_0=klf_0$, by the universal property of the pullback there is an induced morphism $e'_0: D_0 \to B \times B_0$ with $\pi_B e'_0=lf_0$, and then by $(6)$  there is a covering family $\{f_1^{i_1}: D_1^{i_1} \to D_0\}_{i_1 \in I_1}$ and arrows $e'_{i_1}: D_1^{i_1} \to B \times B_0 \times B_1$ with $\pi_{B, B_0}e'_{i_1}=e'_0f_1^{i_1}$. Then we take $e_{i_1}: D_1^{i_1} \to A \times A_0 \times A_1$ as the composite $e_{i_1}=k_1e'_{i_1}$, since with this choice we get $\pi_{A, A_0}e_{i_1}=e_0f_1^{i_1}$. Continuing in this manner we obtain successively covering families from $(6)$, and morphisms $ e_0, e_{i_1}, e_2, ...$ that induce the morphism $e_{\gamma} := (klf_b, e_0\pi_0, e_{i_1}\pi_1, e_2\pi_2, ...): D_{\gamma}^b \to A \times A_{\gamma}$. Now by $(6)$ the morphism $e'_{\gamma} := (lf_b, e'_0\pi_0, e'_{i_1}\pi_1, e'_2\pi_2, ...): D_{\gamma}^b \to B \times B_{\gamma}$ factors through $R$. Since $e_{\gamma}=he'_{\gamma}$ (both give the same morphisms when composed with the projections $\pi_{\alpha}$), it follows that $e_{\gamma}$ factors through $S$, as we wanted. This concludes the proof.
\end{proof}

Dually, in an entirely analogous way we can prove a similar statement for the quantifier $(\exists\forall)_f$. As a consequence, we have:

\begin{cor}\label{pullback}
Given a morphism $f: A \to B$ in a $\kappa$-heterogeneous category $\mathcal{H}$ with respect to a class $\mathcal{C}$, the slices $\mathcal{H}/A$, $\mathcal{H}/B$ are $\kappa$-heterogeneous with respect to $\mathcal{C}$ and the pullback functor $f^*: \mathcal{H}/B \to \mathcal{H}/A$ preserves heterogeneous quantification.
\end{cor}

Finally, we have:

\begin{defs}\label{het}
 A structure in a $\kappa$-heterogeneous category is called well-determined with respect to a class of formulas $\mathcal{C}$ if the interpretation of each $(\forall\exists)_{\alpha<\gamma}\mathbf{x_{\alpha}} \phi$ for $\phi \in \mathcal{C}$ satisfies the axioms $10$ and $11$ of Definition \ref{sfol}, that is, the preservation axioms and the axiom of determinacy.
\end{defs}

We immediately get now:

\begin{lemma}\label{soundness}
$\kappa$-heterogeneous logic (for a class $\mathcal{C}$) is sound with respect to well-determined models in $\kappa$-heterogeneous categories (for the same class $\mathcal{C}$).
\end{lemma}

\begin{proof}
 This follows from soundness of $\kappa$-Heyting logic together with Definition \ref{het} and the straightforward verification that the heterogeneous axioms $10$ of Definition \ref{sfol} are satisfied in any $\kappa$-heterogeneous category.
\end{proof}

From now on we will often omit the reference to the class $\mathcal{C}$, which will be always understood to contain at least the subformulas of the non-logical axioms of the theory to which heterogeneous quantification is applied, and also to be the same in the syntax and the semantic sides.

\subsection{Syntactic $\kappa$-heterogeneous categories and functorial semantics}

The reason why we added the preservation axioms and the axiom of determinacy to the system of $\kappa$-heterogeneous logic (which forced us to restrict the semantics to keep the soundness property) will now become clear when we introduce the syntactic category for that system, since those axioms will enable a functorial interpretation of the semantics in $\kappa$-heterogeneous categories. In other words, we would like to realize the interpretation in $\kappa$-heterogeneous categories functorially by using the usual syntactic construction. Consider, thus, the syntactic category of a theory in $\kappa$-heterogeneous logic, defined as usual but including also formulas involving heterogeneous quantification. We have:

\begin{thm}\label{synt}
The syntactic category $\mathcal{C}_{\mathbb{T}}$ corresponding to a theory in $\kappa$-heterogeneous logic is a $\kappa$-heterogeneous category and its canonical model is well-determined.
\end{thm}

\begin{proof}
Let $A=(\mathbf{x}, \top)$, $A_{\alpha}=(\mathbf{x_{\alpha}}, \top)$. We will show that if the interpretation $[[\phi]]$ of a formula in context $(\mathbf{x} \mathbf{x_{\gamma}}, \phi)$ (where $\phi \in \mathcal{C}$ and $\mathbf{x_{\gamma}}= \cup_{\alpha<\gamma}\mathbf{x_{\alpha}})$ is given by the subobject $(\mathbf{x} \mathbf{x_{\gamma}}, \phi) \rightarrowtail A \times A_{\gamma}$, then the joins $\bigvee_{C \in T_{(\forall\exists)}^{\gamma, [[\phi]]}}C$ and $\bigvee_{C \in T_{(\exists\forall)}^{\gamma, [[\neg \phi]]}}C$ of Definition \ref{het0} exist and are precisely the subobjects $(\mathbf{x}, (\forall\exists)_{\alpha<\gamma}\mathbf{x_{\alpha}} \phi)$ and $(\mathbf{x}, (\exists\forall)_{\alpha<\gamma}\mathbf{x_{\alpha}} \neg \phi)$ of $A$. Then an easy inductive argument shows that these interpretations satisfy axioms $10$ and $11$ of Definition \ref{sfol} with respect to $\mathcal{C}$.

To prove that $\bigvee_{C \in T_{(\forall\exists)}^{\gamma, [[\phi]]}}C = (\forall\exists)_{\alpha<\gamma}\mathbf{x_{\alpha}} \phi$ and $\bigvee_{C \in T_{(\exists\forall)}^{\gamma, [[\neg \phi]]}}C = (\exists\forall)_{\alpha<\gamma}\mathbf{x_{\alpha}} \neg \phi$ as subobjects of $(\mathbf{x}, \top)$, first we show that given $l: C \rightarrowtail (\mathbf{x}, \top)$, if $C \rightarrowtail (\mathbf{x}, (\forall\exists)_{\alpha<\gamma}\mathbf{x_{\alpha}} \phi)$, it is the case that $C \in T_{(\forall\exists)}^{\gamma, [[\phi]]}$. According to Definition \ref{heti}, it is enough to show that:

\begin{itemize}
\item for every arrow $D_0 \to C$ and $e_0: D_0 \to A_0$ there is a cover $D_1 \twoheadrightarrow D_0$ and an arrow $ e_1: D_1 \to A_1$ such that for every arrow $D_2 \to D_1$ and every arrow $e_2: D_2 \to A_2$ there is a cover $D_3 \twoheadrightarrow D_2$ and an arrow $e_3: D_3 \to A_3$ such that $[...]$ such that if $\pi_{\alpha}: D_{\gamma}^b \to D_{\alpha}$ are the limit projections, then the arrow $(lf_b, e_0\pi_0, e_1\pi_1, e_2\pi_2, e_3\pi_3, ...): D_{\gamma}^b \to A \times A_{\gamma}$ factors through $[[\phi]]. \qquad (8)$
\end{itemize}

Let $\pi'_{\alpha}: A \times \Pi_{i<\alpha+1}A_i \to  A \times \Pi_{i<\alpha}A_i$ be the projection, and consider the following statement:

\begin{itemize}
\item Given the pullback $D_0 \to C$ of $\pi'_0$ along $C \rightarrowtail A$ and the composite $e_0: D_0 \to A \times A_0 \to A_0$, there is a cover $D_1 \twoheadrightarrow D_0$ and an arrow $ e_1: D_1 \to A_1$ such that, given the pullback $D_2 \to D_1$ of $\pi'_1$ along $D_1 \rightarrowtail A \times A_0 \times A_1$ and the composite $e_2: D_2 \to A \times A_0 \times A_1 \times A_2 \to A_2$, there is a cover $D_3 \twoheadrightarrow D_2$ and an arrow $e_3: D_3 \to A_3$ such that $[...]$ such that if $\pi_{\alpha}: D_{\gamma}^b \to D_{\alpha}$ are the limit projections, then the arrow $(lf_b, e_0\pi_0, e_1\pi_1, e_2\pi_2, e_3\pi_3, ...): D_{\gamma}^b \to A \times A_{\gamma}$ factors through $[[\phi]]. \qquad (9)$
\end{itemize}

It is not difficult to see (though we leave out the straightforward details) that $(9)$ implies immediately $(8)$, due to the universal property of the pullback, and so we are reduced to verify that $(9)$ holds in $\mathcal{C}_{\mathbb{T}}$. Now, since $C \leq (\forall\exists)_{\alpha<\gamma}\mathbf{x_{\alpha}} \phi$ in $\mathcal{S}ub(A)$, it is the case that the pullback $D_0 \leq (\forall\exists)_{0<\alpha<\gamma}\mathbf{x_{\alpha}} \phi$ in $\mathcal{S}ub(A \times A_0)$, as can be seen by adjunction, after using the first heterogeneous axiom. By the definition of the interpretation of existential quantification, the epi-mono factorization of the composite $\pi'_1 \circ (\forall\exists)_{1<\alpha<\gamma}\mathbf{x_{\alpha}} \phi$ gives precisely the subobject $(\forall\exists)_{0<\alpha<\gamma}\mathbf{x_{\alpha}} \phi$ in $\mathcal{S}ub(A \times A_0)$, and so the pullback of $D_0 \rightarrowtail (\forall\exists)_{0<\alpha<\gamma}\mathbf{x_{\alpha}} \phi$ along the epimorphism of the previous factorization gives a cover $D_1 \twoheadrightarrow D_0$ and and arrow $D_1 \to A \times A_0 \times A_1$ factoring through $(\forall\exists)_{1<\alpha<\gamma}\mathbf{x_{\alpha}} \phi$. We take as $e_1: D_1 \to A_1$ the composite of the arrow $D_1 \to A \times A_0 \times A_1$ with the projection $\pi_{A_1}$. 

We have arrived now at a situation entirely analogous to the beginning: we have $D_1 \leq (\forall\exists)_{1<\alpha<\gamma}\mathbf{x_{\alpha}} \phi$ in $\mathcal{S}ub(A \times A_0 \times A_1)$, and the first two quantifiers of the heterogeneous formula have disappeared. Proceeding in an entirely similar way, we can successively define corresponding covers $D_{\alpha+1} \twoheadrightarrow D_{\alpha}$ for odd $\alpha$ and arrows $E_{\alpha+1}: D_{\alpha+1} \to A_{\alpha+1}$. When reaching a limit ordinal $\beta$, we take the limit $D_{\beta}^l \to C$ along the chain $l: ...D_2 \to D_1 \to D_0 \to C$; then an easy computation shows that the arrow $(f_l, e_0\pi_0, e_1\pi_1, e_2\pi_2,...): D_{\beta} \to A \times A_{\beta} := A \times \Pi_{\alpha<\beta}A_{\alpha}$ is precisely the meet $\bigwedge_{\delta<\beta} (\forall\exists)_{\delta \leq \alpha <\gamma} \mathbf{x_{\alpha}} \phi (\mathbf{x}, \mathbf{x_0}, \mathbf{x_1}, ..., \mathbf{x_{\delta}}, \mathbf{x_{\delta+1}}, ...)$, and the third heterogeneous axiom applied to the limit ordinal $\beta$,  together with the first preservation axiom, imply that this factors through $(\forall\exists)_{\beta \leq \alpha<\gamma} \mathbf{x_{\alpha}} \phi (\mathbf{x}, \mathbf{x_0}, \mathbf{x_1}, ..., \mathbf{x_{\beta}}, \mathbf{x_{\beta+1}}, ...)$. We claim that with this choice of $D_{\alpha}$, for every $\alpha<\gamma$, the conclusion of (5) is satisfied. Indeed, it follows by a similar computation that the arrow $(lf_b, e_0\pi_0, e_1\pi_1, e_2\pi_2, e_3\pi_3, ...): D_{\gamma} \to A \times A_{\gamma}$ is precisely $\bigwedge_{\beta<\gamma} (\forall\exists)_{\beta \leq \alpha<\gamma} \mathbf{x_{\alpha}} \phi (\mathbf{x}, \mathbf{x_0}, \mathbf{x_1}, ..., \mathbf{x_{\beta}}, \mathbf{x_{\beta+1}}, ...)$, and the first preservation axiom now guarantees that this factors through $[[\phi]]$.

Dually, in an entirely analogous way we can prove that given $C \rightarrowtail (\mathbf{x}, (\exists\forall)_{\alpha<\gamma}\mathbf{x_{\alpha}} \neg \phi)$, it is the case that $C \in T_{(\exists\forall)}^{\gamma, [[\neg \phi]]}$. To finish the proof, we need to verify that any $C \in T_{(\forall\exists)}^{\gamma, [[\phi]]}$ must factor through $(\mathbf{x}, (\forall\exists)_{\alpha<\gamma}\mathbf{x_{\alpha}} \phi)$, and an analogous statement for $T_{(\exists\forall)}^{\gamma, [[\neg \phi]]}$. We prove the first, since the argument for the second is similar. Suppose, then that $C \in T_{(\forall\exists)}^{\gamma, [[\phi]]}$ and let us prove that $C \wedge (\mathbf{x}, (\forall\exists)_{\alpha<\gamma}\mathbf{x_{\alpha}} \phi) = C$. By the axiom of determinacy, $C$ must be the join $[C \wedge (\mathbf{x}, (\forall\exists)_{\alpha<\gamma}\mathbf{x_{\alpha}} \phi)] \vee [C \wedge (\mathbf{x}, (\exists\forall)_{\alpha<\gamma}\mathbf{x_{\alpha}} \neg \phi)]$, so we need to prove that $E=C \wedge (\mathbf{x}, (\exists\forall)_{\alpha<\gamma}\mathbf{x_{\alpha}} \neg \phi) = \bot$. But by what we have proved so far, it follows that $E \in T_{(\exists\forall)}^{\gamma, [[\neg \phi]]}$. Since we also have $E \in T_{(\forall\exists)}^{\gamma, [[\phi]]}$, applying both items of Definition \ref{heti} simultaneously, we get the following:

\begin{itemize}
\item there is a covering family $\{D_0^{i_0} \to E\}_{i_0 \in I_0}$ and arrows $e_{i_0}: D_0^{i_0} \to A_0$ such that, for every $i_0 \in I_0$, there is a covering family $\{D_1^{i_0, i_1} \to D_0^{i_0}\}_{i_1 \in I_1}$ and arrows $ e_{i_1}: D_1^{i_0, i_1} \to A_1$ such that, for every $i_0 \in I_0$ and $i_1 \in I_1$ there is a covering family $\{D_2^{i_0, i_1, i_2} \to D_1^{i_0, i_1}\}_{i_2 \in I_2}$ and arrows $e_{i_2}: D_2^{i_0, i_1, i_2} \to A_2$ such that $[...]$ such that if $\pi_{\alpha}: D_{\gamma}^b \to D_{\alpha}^{i_0, i_1, ...}$ are the limit projections and $r: E \rightarrowtail (\mathbf{x}, \top)$, then the arrow $(rf_b, e_{i_0}\pi_0, e_{i_1}\pi_1, e_{i_2}\pi_2, e_{i_3}\pi_3, ...): D_{\gamma}^b \to A \times  A_{\gamma}$ factors through both $[[\phi]]$ and $[[\neg \phi]]. \qquad (10)$
\end{itemize}

Now, by the transfinite transitivity property, it follows that the family of all arrows $(f_b, e_{i_0}\pi_0, e_{i_1}\pi_1, e_{i_2}\pi_2, e_{i_3}\pi_3, ...): D_{\gamma}^b \to A \times  A_{\gamma}$, when $b$ runs over all possible branches of the tree above $E$, is jointly epic. Since each $D_\gamma^b$ forces $\bot$, we conclude that $E=\bot$, as we wanted to prove.
\end{proof}

As an immediate consequence we get:

\begin{cor}\label{compc}
$\kappa$-heterogeneous logic is complete with respect to well-determined models in $\kappa$-heterogeneous categories.
\end{cor}

The construction of the $\kappa$-heterogeneous syntactic category enables a functorial interpretation of semantics. Given a well-determined model $M$ of a theory $\mathbb{T}$ in a $\kappa$-heterogeneous category $\mathcal{D}$, there is an evident functor $F_M: \mathcal{C}_{\mathbb{T}} \to \mathcal{D}$ whose object part send the formula in context $(\mathbf{x}, \phi)$ to the interpretation $[[\phi(\mathbf{x})]]^{M}$ of $\phi(\mathbf{x})$ in $M$. Then $F_M$ is clearly a $\kappa$-Heyting functor that preserves heterogeneous quantification, that is: 

$$[[(\forall\exists)_{\alpha<\gamma}\mathbf{x_{\alpha}} \phi(\mathbf{x}, \mathbf{x_{\alpha}})]]^M = \bigvee_{C \in T_{(\forall\exists)}^{\gamma, [[\phi]]^M}}C$$

$$[[(\exists\forall)_{\alpha<\gamma}\mathbf{x_{\alpha}} \neg \phi(\mathbf{x}, \mathbf{x_{\alpha}})]]^M = \bigvee_{C \in T_{(\exists\forall)}^{\gamma, [[\neg \phi]]^M}}C$$
\noindent

But more is true: we shall prove that any $\kappa$-Heyting functor $F_M: \mathcal{C}_{\mathbb{T}} \to \mathcal{D}$ to a $\kappa$-heterogeneous category must automatically preserve heterogeneous quantification, as a direct consequence of the axiom of determinacy and the preservations axioms (whence the name of the latter). This is the non-trivial part of the following:

\begin{thm}\label{funct}
Let $(A_{\alpha})_{\alpha<\kappa}$ witness that a category $\mathcal{D}$ is $\kappa$-heterogeneous. Then well-determined models $M$  in $\mathcal{D}$ of a $\kappa$-heterogeneous theory $\mathbb{T}$ correspond precisely to $\kappa$-Heyting functors $F_M: \mathcal{C}_{\mathbb{T}} \to \mathcal{D}$ such that $F((\mathbf{x_{\alpha}}, \top))=A_{\alpha}$.
\end{thm}

\begin{proof}
Suppose we have a $\kappa$-Heyting functor $F: \mathcal{C}_{\mathbb{T}} \to \mathcal{D}$ to a $\kappa$-Heyting category such that $A_{\alpha}=F((\mathbf{x_{\alpha}}, \top))$, and let $a: C \to A$. Using successively the heterogeneous axioms together with the fact that $F_M$ is $\kappa$-Heyting, we get the following implications:

$$C \Vdash F((\forall\exists)_{\alpha<\gamma}\mathbf{x_{\alpha}} \phi (a, \mathbf{x_0}, \mathbf{x_1}, ...))$$

$$\implies \text{for all $D_0 \to C, e_0: D_0 \to A_0$ we have that $D_0 \Vdash F((\forall\exists)_{0<\alpha<\gamma}\mathbf{x_{\alpha}} \phi (a, e_0, \mathbf{x_1}, ...))$}$$

$$\implies \text{there is a covering family $\{D_1^{i_1} \to D_0\}_{i_1 \in I_1}$ and arrows $ e_{i_1}: D_1^{i_1} \to A_1$} $$
$$\text{such that for every $i_1 \in I_1$ we have $D_1^{i_1} \Vdash F((\forall\exists)_{1<\alpha<\gamma}\mathbf{x_{\alpha}} \phi (a, e_0, e_{i_1}, ...))$}$$

$$\implies [...]$$

$$\implies \text{the limit $D_{\gamma}^b$ satisfies $D_{\gamma}^b \Vdash F\left(\bigwedge_{\beta<\gamma} (\forall\exists)_{\beta \leq \alpha<\gamma} \mathbf{x_{\alpha}} \phi (a f_b, e_0 \pi_0, e_{i_1} \pi_1, ..., \mathbf{x_{\beta}}, ...)\right)$}$$

By the first preservation axiom, it follows that $D_{\gamma}^b \Vdash F(\phi (a f_b, e_0 \pi_0, e_{i_1} \pi_1, ...))$, and so, by definition, that $C \in T_{(\forall\exists)}^{\gamma, F(\phi (a, ...))}$. Whence, $C \Vdash \bigvee_{C \in T_{(\forall\exists)}^{\gamma, F(\phi(a, ...))}}C$. Therefore, since $a: C \to A$ was arbitrary, it follows that $F((\forall\exists)_{\alpha<\gamma}\mathbf{x_{\alpha}} \phi(\mathbf{x}, \mathbf{x_{\alpha}})) \leq \bigvee_{C \in T_{(\forall\exists)}^{\gamma, F(\phi)}}C$.

Dually, an analogous argument shows that $F((\exists\forall)_{\alpha<\gamma}\mathbf{x_{\alpha}} \neg \phi(\mathbf{x}, \mathbf{x_{\alpha}})) \leq \bigvee_{C \in T_{(\exists\forall)}^{\gamma, F(\neg \phi)}}C$. By the axiom of determinacy, the union of the subobjects $F((\forall\exists)_{\alpha<\gamma}\mathbf{x_{\alpha}} \phi(\mathbf{x}, \mathbf{x_{\alpha}}))$ and  $F((\exists\forall)_{\alpha<\gamma}\mathbf{x_{\alpha}} \neg \phi(\mathbf{x}, \mathbf{x_{\alpha}}))$ is the whole of $A$. But the subobjects $\bigvee_{C \in T_{(\forall\exists)}^{\gamma, F(\phi)}}C$ and $\bigvee_{C \in T_{(\exists\forall)}^{\gamma, F(\neg \phi)}}C$ are disjoint, as can be seen by applying to their intersection $E$ the same argument as in statement $(10)$ in the proof of Theorem \ref{synt}. Hence, this forces the sign $\leq$ to be an equality in both cases, as we wanted to prove.
\end{proof}

It follows in particular that when considering the topology $\tau$ on the syntactic category $\mathcal{C}_{\mathbb{T}}$ given by jointly epic morphisms of cardinality less than $\kappa^+$, Yoneda embedding $Y: \mathcal{C}_{\mathbb{T}} \to \mathcal{S}h(\mathcal{C}_{\mathbb{T}}, \tau)$, being $\kappa$-Heyting (see \cite{espindola}), preserves heterogeneous quantification. This gives us the following:

\begin{cor}\label{compt}
$\kappa$-heterogeneous logic is complete with respect to well-determined models in $\kappa$-Grothendieck toposes.
\end{cor}

\section{Completeness}

The techniques so far have exploited the preservation axiom and the axiom of determinacy to deal with heterogeneous quantification. The same arguments, that should be by now familiar to the reader, can be used to prove several completeness results, as we now proceed to do.

\subsection{Completeness of classical $\kappa$-heterogeneous logic}

We have so far established completeness theorems for $\kappa$-heterogeneous logic in terms of well-determined models in $\kappa$-heterogeneous categories, and even in (more specifically) $\kappa$-Grothendieck toposes. We will now proceed to study completeness in terms of well-determined set-valued models. Naturally, it is not possible to have such a theorem since the law of excluded middle (i.e., the sequents $\top \vdash_{\mathbf{x}} \phi(\mathbf{x}) \vee \neg \phi(\mathbf{x})$) is valid in set models, though is not part of the axioms. However, we will see that this is the only constraint. As soon as we add all instances of excluded middle, getting thus classical $\kappa$-heterogeneous logic, valid sentences in well-determined structures are provable from the axioms. 

In the same way as $\kappa$-first-order logic can be \emph{Morleyized} (i.e., we can define a Morita-equivalent theory in a less expressive fragment) to get a theory in $\kappa$-coherent logic (see \cite{espindola}), so $\kappa$-heterogeneous logic has its own Morleyization. The formulas of this latter are obtained inductively by allowing conjunctions of less than $\kappa$ many subformulas, disjunctions of less than $\kappa^+$ many subformulas, and existential quantification on less than $\kappa$ many variables. Given a theory $\mathbb{T}$, we denote by $S$ the set of subformulas of antecedents and consequents of the axioms of $\mathbb{T}$ (we include in $S$ even subformulas of each $(\forall\exists)_{\alpha<\gamma}\mathbf{x_{\alpha}} \phi$ of the form $(\forall\exists)_{\beta \leq \alpha<\gamma}\mathbf{x_{\alpha}} \phi$, and analogously with quantifiers of the type $(\exists\forall)$). The corresponding Morleyized heterogeneous theory is defined in the following:

\begin{defs}
Given a classical $\kappa$-heterogeneous theory $\mathbb{T}$, we define its Morleyized theory $\theory^m$, over a signature $\Sigma^m$ that extends the original signature $\Sigma$ by adding for each $\kappa$-first-order formula $\phi \in S$ over $\Sigma$ with free variables \alg{x} two new relation symbols $C_{\phi}(\alg{x})$ and $D_{\phi}(\alg{x})$, and whose axioms are:

\begin{enumerate}[(i)]

\item $C_{\phi} \wedge D_{\phi} \vdash_{\alg{x}} \bot$;
\item $\top \vdash_{\alg{x}} C_{\phi} \vee D_{\phi}$;
\item $C_{\phi}\dashv\vdash_{\alg{x}}\phi$ for every atomic formula $\phi$;
\item $C_{\phi}\vdash_{\alg{x}}C_{\psi}$ for every axiom $\phi\vdash_{\alg{x}}{\psi}$ of \theory (including logical axioms);
\item $D_{\bigwedge_{i<\gamma}\phi_i}\dashv\vdash_{\alg{x}}\bigvee_{i<\gamma}D_{\phi_i}$;
\item $C_{\bigvee_{i<\gamma}\phi_i}\dashv\vdash_{\alg{x}}\bigvee_{i<\gamma}C_{\phi_i}$;
\item $C_{\phi \rightarrow \psi}\dashv\vdash_{\alg{x}}D_{\phi} \vee C_{\psi}$
\item $C_{\exists{\mathbf{y}}\phi}\dashv\vdash_{\alg{x}}\exists{\mathbf{y}}C_{\phi}$;
\item $D_{\forall{\mathbf{y}}\phi}\dashv\vdash_{\alg{x}}\exists{\mathbf{y}}D_{\phi}$;
\item $D_{(\forall\exists)_{\alpha<\gamma}{\mathbf{x_{\alpha}} \phi}}\dashv\vdash_{\alg{x}}\exists{\mathbf{x_{0}}}D_{(\forall\exists)_{0<\alpha<\gamma}{\mathbf{x_{\alpha}} \phi}}$ \qquad for each $\phi \in \mathcal{C}$;
\item $C_{(\exists\forall)_{\alpha<\gamma}{\mathbf{x_{\alpha}} \neg \phi}}\dashv\vdash_{\alg{x}}\exists{\mathbf{x_{0}}}C_{(\exists\forall)_{0<\alpha<\gamma}{\mathbf{x_{\alpha}} \neg \phi}}$ \qquad for each $\phi \in \mathcal{C}$.
\end{enumerate} 
\end{defs}

Note that these axioms are all $\kappa$-coherent and they ensure that the interpretations of $(\alg{x}, C_{\phi}(\alg{x}))$ and $(\alg{x}, D_{\phi}(\alg{x}))$ in any Boolean $\kappa$-heterogeneous category (including \Sets) will coincide with those of $(\alg{x}, \phi(\alg{x}))$ and $(\alg{x}, \neg \phi(\alg{x}))$, respectively, and that, moreover, $\theory^m$-models coincide with $\theory$-models in such categories. Furthermore, the syntactic categories $\mathcal{C}_{\mathbb{T}}$ and $\mathcal{C}_{\theory^m}$ are equivalent, as can be easily checked, and hence $\mathcal{C}_{\theory^m}$ will be a classical $\kappa$-heterogeneous category. We now have:

\begin{thm}
Classical $\kappa$-heterogeneous logic is complete with respect to set-valued models.
\end{thm}

\begin{proof}
Suppose now that $\mathbb{T}$ is a theory with at most $\kappa$ many axioms, and let $\phi \vdash_{\alg{x}} \psi$ be a sequent valid in well-determined $\mathbb{T}$-models in \Sets; it follows that $C_{\phi}\vdash_{\alg{x}}C_{\psi}$ will be valid in every $\theory^m$-model in \Sets. We shall show that then the sequent will be provable in $\theory^m$. Assuming we have done that, now replace in this proof every subformula of the form $C_{\phi}(t_1, ..., t_{\alpha}, ...)$ by the corresponding substitution instance of $\phi(\alg{t}/\alg{x})$, and every subformula of the form $D_{\phi}(t_1, ..., t_{\alpha}, ...)$ by the corresponding substitution instance of $\neg \phi(\alg{t}/\alg{x})$. We claim that this way we will get a proof in $\theory$ of the sequent $\phi \vdash_{\alg{x}} \psi$ using the rules of $\kappa$-first-order systems. Indeed, the effect of the transformation just described on the axioms of $\theory^m$ produces either axioms of $\theory$ or sequents that are classically provable from the axioms of $\theory$ and the logical axioms of classical $\kappa$-heterogeneous logic. Therefore, to prove completeness of classical $\kappa$-heterogeneous logic it is enough to show that there is a jointly conservative family of well-determined $\theory^m$-models in \Sets. But this is a consequence of the completeness theorem for $\kappa$-coherent theories (the $(\kappa^+, \kappa, \kappa)$-coherent fragment of \cite{espindola}). This finishes the proof.
\end{proof}

\begin{rmk}
It follows from the axiomatization of the Morleyized theory that the category of well-determined models of the theory is accessible, being equivalent to the category of models of a $\kappa$-coherent theory.
\end{rmk}

\subsection{Completeness of intuitionistic $\kappa$-heterogeneous logic over $\mathcal{L}_{\kappa^+, \kappa, \kappa}$}

The completeness theorem for intuitionistic first-order logic over $\mathcal{L}_{\kappa^+, \kappa, \kappa}$ (see \cite{espindola2}) can be adapted to the heterogeneous setting, since instances of excluded middle are not necessarily part of the axioms, except for those instances that correspond to instances of determinacy. In particular, any intuitionistic first-order heteerogeneous theory of cardinality at most $\kappa$ over $\mathcal{L}_{\kappa^+, \kappa, \kappa}$ admits the following Morleyization:

\begin{enumerate}[(i)]

\item $C_{(\forall\exists)_{\alpha<\gamma}\mathbf{x_{\alpha}} \phi} \wedge D_{(\forall\exists)_{\alpha<\gamma}\mathbf{x_{\alpha}} \phi} \vdash_{\alg{x}} \bot$;
\item $C_{(\exists\forall)_{\alpha<\gamma}\mathbf{x_{\alpha}} \neg \phi} \wedge D_{(\exists\forall)_{\alpha<\gamma}\mathbf{x_{\alpha}} \neg \phi} \vdash_{\alg{x}} \bot$;
\item $\top \vdash_{\alg{x}} C_{(\forall\exists)_{\alpha<\gamma}\mathbf{x_{\alpha}} \phi} \vee D_{(\forall\exists)_{\alpha<\gamma}\mathbf{x_{\alpha}} \phi}$;
\item $\top \vdash_{\alg{x}} C_{(\exists\forall)_{\alpha<\gamma}\mathbf{x_{\alpha}} \neg \phi} \vee D_{(\exists\forall)_{\alpha<\gamma}\mathbf{x_{\alpha}} \neg \phi}$;
\item $C_{\phi}\dashv\vdash_{\alg{x}}\phi$ for every atomic formula $\phi$;
\item $C_{\phi}\vdash_{\alg{x}}C_{\psi}$ for every axiom $\phi\vdash_{\alg{x}}{\psi}$ of \theory (including logical axioms);
\item $D_{\bigwedge_{i<\gamma}\phi_i}\dashv\vdash_{\alg{x}}\bigvee_{i<\gamma}D_{\phi_i}$;
\item $C_{\bigvee_{i<\gamma}\phi_i}\dashv\vdash_{\alg{x}}\bigvee_{i<\gamma}C_{\phi_i}$;
\item $C_{\exists{\mathbf{y}}\phi}\dashv\vdash_{\alg{x}}\exists{\mathbf{y}}C_{\phi}$;
\item $D_{(\forall\exists)_{\alpha<\gamma}{\mathbf{x_{\alpha}} \phi}}\dashv\vdash_{\alg{x}}\exists{\mathbf{x_{0}}}D_{(\forall\exists)_{0<\alpha<\gamma}{\mathbf{x_{\alpha}} \phi}}$ \qquad for each $\phi \in \mathcal{C}$;
\item $C_{(\exists\forall)_{\alpha<\gamma}{\mathbf{x_{\alpha}} \neg \phi}}\dashv\vdash_{\alg{x}}\exists{\mathbf{x_{0}}}C_{(\exists\forall)_{0<\alpha<\gamma}{\mathbf{x_{\alpha}} \neg \phi}}$ \qquad for each $\phi \in \mathcal{C}$.
\end{enumerate}

Set-valued models of the theory above correspond to $\kappa$-coherent functors from $\mathcal{C}_{\mathbb{T}}$ and are therefore well-determined structures in $\mathcal{S}et$. Considering now the evaluation functor $ev: \mathcal{C}_{\mathbb{T}} \to \mathcal{S}et^{\mathcal{C}oh(\mathbb{T})}$, where $\mathcal{C}oh(\mathbb{T})$ is a suitably full subcategory of coherent models and homomorphisms, Theorem \ref{funct} guarantes that $ev$ preserves heterogeneous quantification, while by the same arguments used in the proof of Theorem 2.3.2 from \cite{espindola2} we can now conclude the following:

\begin{thm}
 If $\kappa$ is regular and $\kappa^{<\kappa}=\kappa$, intuitionistic heterogeneous (with respect to a class $\mathcal{C}$) theories over $\mathcal{L}_{\kappa^+, \kappa, \kappa}$, of cardinality at most $\kappa$, are complete with respect to Kripke models, where we extend the notion of forcing adding the following clauses for heterogeneous quantification:

\begin{enumerate}
\item $p \Vdash (\forall\exists)_{\alpha<\gamma}\mathbf{x_{\alpha}} \phi(\mathbf{c})  \iff M_p \models (\forall\exists)_{\alpha<\gamma}\mathbf{x_{\alpha}} (p \Vdash \phi(\mathbf{c}))$
\item $p \Vdash (\exists\forall)_{\alpha<\gamma}\mathbf{x_{\alpha}} \neg \phi(\mathbf{c})  \iff M_p \models (\exists\forall)_{\alpha<\gamma}\mathbf{x_{\alpha}} (p \Vdash \neg \phi(\mathbf{c}))$
\end{enumerate}

where $M_p$ is the underlying structure of the node $p$ and $\phi \in \mathcal{C}$.
\end{thm}

\subsection{Completeness of $\kappa$-heterogeneous logic with bounded quantifiers}

The definition of heterogeneous quantification in a Grothendieck topos required a type of heterogeneous quantification in the metatheory in which all quantifiers appeared bounded; it is thus natural to expect that the definition can be actually strengthened to cover this case. In fact, all the development of heterogeneous quantification so far can be cast, mutatis mutandi, to a setting in which we define a special quantification (where quantifiers appear bounded) denoted as $(\forall\exists^{(\psi_{\alpha})})_{\alpha<\gamma} \mathbf{x_{\alpha}} \phi$, and whose intended meaning is the following:

$$\forall \mathbf{x_0} (\psi_0(\mathbf{x_0}) \to \exists \mathbf{x_1} (\psi_1(\mathbf{x_1}) \wedge (\forall \mathbf{x_2} (\psi_2(\mathbf{x_2}) \to ... \phi (\mathbf{x_0}, \mathbf{x_1}, \mathbf{x_2}, ...)...)$$
\\
proceeding similarly with the dual quantification. The axioms of heterogeneous logic are the same for this version, except that we have to adopt the following modifications to the heterogeneous axioms:

$$(\forall\exists^{(\psi_{\alpha})})_{\alpha<\gamma} \mathbf{x_{\alpha}} \phi \vdash_{\mathbf{x}}\forall \mathbf{x_0} (\psi_0(\mathbf{x_0}) \to (\forall\exists^{(\psi_{\alpha})})_{0< \alpha<\gamma} \mathbf{x_{\alpha}} \phi)$$
$$(\exists\forall^{(\psi_{\alpha})})_{\alpha<\gamma} \mathbf{x_{\alpha}} \neg \phi \vdash_{\mathbf{x}}\exists \mathbf{x_0} (\psi_0(\mathbf{x_0}) \wedge (\exists\forall^{(\psi_{\alpha})})_{0<\alpha<\gamma} \mathbf{x_{\alpha}} \neg \phi)$$

This results in a more expressive language for which we can prove analogously the completeness theorems much as before, with respect to heterogeneous categories, Grothendieck toposes, set and Kripke models. Classically this type of quantification is expressible as a Vaught sentence (see \cite{vaught}), but when considering the intuitionistic case, the implication $\to$ ceases to be definable from the rest of the connectives, and this quantification becomes really a new type of expression in which the usage of $\to$ appears infinitely deep in the formula. 

\section{Acknowledgements}

This research has been supported through the grants P201/12/G028 and 19-00902S from the Grant Agency of the Czech Republic.

\bibliographystyle{amsalpha}

\renewcommand{\bibname}{References} % changes the header from Bibliography to References

\bibliography{references}

%\begin{thebibliography}{widest entry}

%\end{thebibliography}

\end{document}